\documentclass{amsart}

\usepackage{amssymb}

\usepackage{graphicx}

\usepackage{url}

\usepackage{color}
\newtheorem{theorem}{Theorem}[section]
\newtheorem{lemma}[theorem]{Lemma}

\theoremstyle{definition}
\newtheorem{definition}[theorem]{Definition}

\theoremstyle{remark}
\newtheorem{remark}[theorem]{Remark}

\numberwithin{equation}{section}

\newtheorem{corollary}[theorem]{Corollary}
\newtheorem{proposition}[theorem]{Proposition}

\def\F2{{\mathbb F}_2}

\def\bR{{\mathbb R}}

\def\bb{{{\mathbf b}}}

\def\bk{{{\mathbf k}}}
\def\bx{{{\mathbf x}}}
\def\by{{{\mathbf y}}}

\def\b1{{{\mathbf 1}}}

\def\cP{{{\mathcal P}}}
\def\WF{{{\mathrm{WF}}}}
\def\Error{{{\mathrm{Err}}}}
\def\Vol{{{\mathrm{Vol}}}}

\begin{document}
\title{A computable figure of merit for quasi-Monte Carlo point sets}
\author{MAKOTO MATSUMOTO}
\address{Graduate School of Mathematical Sciences, 
University of Tokyo, Tokyo 153-8914 Japan}
\email{matumoto@ms.u-tokyo.ac.jp}
\thanks{Partially
supported by JSPS/MEXT Grant-in-Aid for Scientific Research
No.23244002, No.19204002, No.21654017, and JSPS Core-to-Core Program
No.18005.}
\author{MUTSUO SAITO}
\address{Department of Mathematics, Hiroshima University, Hiroshima
739-8526 Japan}
\email{saito@math.sci.hiroshima-u.ac.jp}
\thanks{Partially
supported by JSPS/MEXT Grant-in-Aid for Scientific Research 
No.21654004}
\author{KYLE MATOBA}
\address{Finance Department, UCLA Anderson School of Management, Los Angeles, USA}
\email{kyle.matoba.2014@anderson.ucla.edu}

\subjclass[2010]{Primary 11K38, 11K45, 65C05; Secondary  65D30, 65T50}
\date{\today}
\keywords{quasi-Monte Carlo, numerical integration, 
digital nets,
low discrepancy sequences, 
Walsh functions,
figure of merit, WAFOM, computational finance.
}

\begin{abstract}
Let $\cP \subset [0,1)^S$ be a finite point set of cardinality $N$
in an $S$-dimensional cube, 
and let $f:[0,1)^S \to \bR$
be an integrable function. A QMC integration of $f$
by $\cP$ is the average of values of $f$ at each point
in $\cP$, which approximates the integration of $f$
over the cube.
Assume that $\cP$ is constructed from 
an $\F2$-vector space $P\subset (\F2^n)^S$ by means of a digital net
with $n$-digit precision.
As an $n$-digit discretized version of Josef Dick's method,
we introduce Walsh figure of merit (WAFOM) $\WF(P)$ of $P$, 
which satisfies a Koksma-Hlawka type inequality, namely, 
QMC integration error is bounded by 
$C_{S,n}||f||_n \WF(P)$ under $n$-smoothness of $f$,
where $C_{S,n}$ is a constant depending only 
on $S,n$.

We show a Fourier inversion formula for $\WF(P)$
which is computable in $O(n SN)$ steps. This effectiveness enables us 
a random search for $P$ with small value of $\WF(P)$,
which would be difficult for other figures of merit such as
discrepancy. From an analogy to coding theory, we expect
that random search may find better point sets than mathematical
constructions. In fact, a na\"{i}ve search finds point sets $P$
with small $\WF(P)$. 
In experiments, we show better performance of these point sets
in QMC integration than widely used QMC rules. 
We show some experimental evidence
on the effectiveness of our point sets to even non-smooth integrands
appearing in finance.
\end{abstract}

\maketitle

%
%
\section{Introduction}
\label{sec:introduction}

There is a strong analogy between coding theory and 
quasi-Monte Carlo (QMC) point sets (e.g., see 
\cite{SKRIGANOV01, SKRIGANOV06} \cite{CHEN-SKRIGANOV}
\cite{niederreiter:coding}).
In coding theory, it seems to be widely believed that 
to find good codes, a random search is easier than
mathematical explicit constructions. This is supported
by the fundamental theorem of 
Shannon showing that a randomly chosen code is 
good with high probability \cite{shannon1948}.
Further support is
the success of low density parity-check 
(LDPC) codes 
\cite{MacKay1999}. 
Thus, it is natural 
to consider a random search for good QMC point sets.
An obstruction is the 
lack of a {\em practically} computable measure for the ``goodness''
of QMC point sets. In coding theory, 
the minimum distance is a decisive and computable 
measure for the quality of the codes, so a random
search to optimize this value works well, if the cardinality 
of the code words is not large. 
For QMC point sets, 
the star-discrepancy (\cite{niederreiter:book} \cite{DICK-PILL-BOOK})
is hard to compute. Alternative
diaphony \cite{Zinterhof} and dyadic diaphony \cite{Hellekalek}
require $O(SN^2)$ operations, where $S$ is the dimension
and $N$ is the cardinality of the point sets, 
which would be practically difficult for large $N$ for a random search. 
Probably, the most important figure of merit on digital nets
is the $t$-value of
$(t,m,s)$-net introduced by Niederreiter
(\cite{niederreiter:book} \cite{DICK-PILL-BOOK}).
The $t$-value gives an upper bound on
star-discrepancy.
Its possible weakness is that $t$ assumes only an integer value, and thus may be 
too coarse to give a tight bound on the star-discrepancy.

In this paper we propose {\em Walsh figure of merit}
(WAFOM) as a practically computable measure of goodness of a QMC point
set from a digital net \cite[\S4.3]{niederreiter:book}.
An inversion formula (\ref{eq:WAFOM}) computes WAFOM 
in only $O(nSN)$ steps, where $n$ is the precision 
(see \S\ref{sec:discretization}) of
the point sets.
We can easily execute a random search minimizing
WAFOM. Our experiment yields QMC point sets which perform better than
some standard low discrepancy point sets, for QMC integration 
of a natural integrand
arising in computational finance.

Let us briefly recall the notion of QMC integration.
Let $\cP \subset [0,1)^S$ be a point set in 
an $S$-dimensional cube with finite cardinality 
$|\cP|=N$, and let $f:[0,1)^S \to \bR$
be an integrable function.
The quasi-Monte Carlo integration by $\cP$
is an approximation value
\begin{equation} \label{eq:QMC-integral}
I_\cP(f):=\frac{1}{|\cP|} \sum_{\bx \in \cP} f(\bx)
\end{equation}
of the actual integration
\begin{equation} \label{eqn:integral}
I(f):=\int_{[0,1)^S}f(\bx)d\bx.
\end{equation}
Thus, the QMC integration error is 
\begin{equation}\label{eq:QMC-error}
\Error(f; \cP):= |I(f)-I_\cP(f)|.
\end{equation}

A central idea in the theory of QMC
is to show an upper bound on this error of the form
$$
\Error(f; \cP) \leq V(f) D(\cP),
$$
where $V(f)$ is a quantity depending only on $f$
(not $\cP$) and $D(\cP)$ is a quantity depending
only on $\cP$, for a suitable class of integrand functions.
Then, one designs $\cP$ with a small value of $D(\cP)$,
which works for this class of functions.
In the case of the monumental Koksma-Hlawka inequality, 
$V(f)$ is the total variation of $f$,
$D(\cP)$ is the star-discrepancy of $\cP$, 
and the class of functions
is those with bounded variation.
In this case, there are many studies 
on the construction of point sets with small $D(\cP)$, 
in particular satisfying a conjectured lower bound 
of $O(N^{-1}(\log N)^{S-1})$ (see \cite{niederreiter:book}
\cite{DICK-PILL-BOOK} for these basics).

For a given integer $\alpha>1$, 
Dick \cite{Dick-Walsh} \cite{Dick-MCQMC} introduced 
the notion of $\alpha$-smooth functions, the 
$\alpha$-norm $||f||_\alpha$, and (he did not give a name
so we term it here as) Walsh figure-of-merit (WAFOM) 
$\WF_\alpha(\cP)$, so that
$$
\Error(f; \cP) \leq C_{S,n} ||f||_\alpha \cdot \WF_\alpha(\cP)
$$
holds, using Walsh functions (see \cite[\S4.2]{Dick-MCQMC})
for a class of point sets
(i.e., digital-nets). 
Dick constructed families of point sets 
with $\WF_\alpha(\cP)\leq O(N^{-\alpha}(\log N)^{\alpha S})$.
A difficulty is 
that to realize a large value of $\alpha$, $N$ becomes 
too large for practical QMC integration.

Our approach is as follows.
In \S\ref{sec:discretization}, we fix an integer $n$ 
called the {\em degree of discretization},
discretize $[0,1)$ into $2^{n}$ intervals, and 
approximate $f$ by a function $f_n$
which is constant on each small cube.
If we assume Lipschitz continuity of $f$, the approximation
error is $O(2^{-n})$.
In \S\ref{sec:preliminary}, we recall discrete Fourier transformation.
In \S\ref{sec:WAFOM}, 
we introduce an $n$-digit
discretized version $\WF(P)$ of $\WF_\alpha(\cP)$,
and present a result of Dick bounding
the integration error by WAFOM  in our discretized
context. In \S\ref{sec:explicit}, we give a Fourier inversion formula
for WAFOM 
that reduces the computational complexity to $O(nSN)$.
Using this formula, we search for point sets with
small value of WAFOM by random search.
In \S\ref{sec:experiments}, we show results of numerical experiments.
We find point sets with small $\WF(P)$: the values suggest that 
$\WF(P)=O(N^{-1-\beta})$ for some positive constant $\beta$
near one, even in a practical range of $N=2^{10}, 2^{11},\ldots, 2^{22}$.
Experimentation 
shows that these point sets perform better than some standard
low discrepancy sequences. 

\section{Discretization}\label{sec:discretization}
Let us fix a positive integer $n$, which is 
called the degree of discretization.
To simplify the analysis, we discretize
$I:=[0,1)$ into a finite set with cardinality $2^n$. 
It would be possible to do a similar analysis without such discretization
using Walsh functions \cite{Dick-MCQMC, Dick-Walsh}
\cite{DICK-PILL-BOOK}
\cite{Hellekalek}, but we do not pursue it here
and work only on the discrete case because 
we are interested in applications in digital computers
where such discretization is implicitly done.

The interval $I$ is decomposed into $2^n$ intervals
of equal length $[b, b+2^{-n})$, where 
\begin{equation}\label{eq:binary-fraction}
b = b_12^{-1}+\cdots +b_n2^{-n}, \quad b_1,\ldots,b_n \in \{0,1\}.
\end{equation}
Such a $b$ is called an {\em $n$-digit binary fraction}.

Let $\F2=\{0,1\}$ be the two-element field.
In $\F2$, arithmetic operations are done modulo 2.
Let $\F2^n$ be the space of $n$-dimensional row vectors:
$$
\F2^n = \{(b_1,\ldots,b_n) \ | \ b_i \in \F2 \}.
$$
The sum of two such vectors is computed componentwise
modulo 2.

There is a natural identification of a vector
$(b_1,\ldots,b_n)$ with an 
$n$-digit binary fraction $b$ as in (\ref{eq:binary-fraction}),
hence we have an identification of the three sets:
$\F2^n$, the set of $n$-digit binary fractions, and the set of
the $2^n$ intervals, by
$$
 (b_1,\ldots,b_n) 
 \mapsto b= b_12^{-1}+\cdots +b_n2^{-n}
 \mapsto I_b:=[b, b+2^{-n}).
$$

Let $V:=(\F2^n)^S$ be the set of $S \times n$ matrices with 
components in $\F2$. As usual, $B \in V$ is described as 
$B:=(b_{T,j})_{1\leq T \leq S, 1\leq j \leq n}$, with $b_{T,j}\in \F2$. To $B$, we associate $\bb=(b_1,\ldots, b_S) \in I^S$
with $b_T=b_{T,1}2^{-1}+\cdots+b_{T,n}2^{-n}$,
and then an $S$-dimensional cube
$I_B:=I_\bb := I_{b_1}\times I_{b_2} \times \cdots \times I_{b_S}$.
This gives a discretization of $I^S$ by $V=(\F2^n)^S$.

Let $f:I^S \to \bR$ be an integrable function.
We define its $n$-digit discrete approximation $f_n$ as 
$$
f_n: V \to \bR, \quad
f_n(B):=\frac{1}{\Vol(I_\bb)}\int_{I_\bb} f(\bx) d\bx.
$$
In other words, to $B \in V$, we associate
an $S$-dimensional cube consisting of $(x_1,\ldots,x_S)$
such that the first $n$ digits of 
the binary expansion of $x_T$ being equal to $(b_{T,1},\ldots,b_{T,n})$
for each $T$,
and $f_n(B)$ is the average value of $f$ over this cube.

It is easy to see
\begin{equation}\label{eq:integration-noerror}
\frac{1}{|V|}\sum_{B \in V} f_n(B)
=
\int_{I^S} f(\bx) d\bx.
\end{equation}

In the following, we assume that
each point in $\cP$
has coordinates in $n$-digit binary fractions.
Thus, $\cP \subset I^S$ corresponds to a subset $P \subset V$. 

We define the {\em $n$-th discretized} QMC integration of $f$ by $P$ as 
\begin{equation}\label{eq:discretized-QMC}
I_{P,n}(f):=\frac{1}{|P|}\sum_{B \in P}f_n(B).
\end{equation}

Note that this value is hard to compute in practice, since
each $f_n(B)$ is an integration,
but can be approximated by the usual QMC integration $I_{\cP}(f)$ 
as follows.

Let us define
$$
||f-f_n||_\infty:=
\sup_{B\in V, \bb \in I_B} |f_n(B)-f(b)|.
$$
This coincides with the usual supremum norm if $f_n$ is regarded as 
a (not necessarily continuous) function
on $I^S$ that takes the (constant) value $f_n(B)$ on the
hypercube $I_B$. From now on, we use the same notation 
$f_n$ for this function on $I^S$.

\begin{lemma}
\label{lemma:QMC_discretization}
The difference between the QMC integral (\ref{eq:QMC-integral})
and the discretized QMC integral
(\ref{eq:discretized-QMC})
is bounded by
$$
|I_\cP(f)-I_{P,n}(f)| \leq ||f-f_n||_\infty. 
$$
In particular, if $f$ is continuous with 
Lipschitz constant $K$, namely,
if for any $\bx,\bx' \in [0,1)^S$
$$
|f(\bx)-f(\bx')| \leq K ||\bx-\bx'||
$$
holds (where the right hand side is the Euclidean norm), 
then the above value is bounded by
$$||f-f_n|| \leq K\sqrt{S}2^{-n}.$$
\end{lemma}
\begin{proof}
The left hand side of the first inequality
is bounded by the average of $|f(\bb)-f_n(\bb)|$
over $\bb \in \cP$, hence is bounded by the supremum norm
$||f-f_n||_\infty$.
For the second inequality, for any $\bb$ take $I_B$
that contains $\bb$. Then, 
since the diameter of $I_B$ is $\sqrt{S}2^{-n}$, 
$$
|f(\bb)-f_n(\bb)|\leq \sup_{\bx,\by \in I_B} |f(\bx)-f(\by)|
\leq K\cdot 2^{-n}\sqrt{S}.
$$
\end{proof}
By this lemma, when we take $n$ large enough so that
$K\cdot 2^{-n}\sqrt{S}$ is negligible compared to
the QMC integration error, then we may identify
$I_{P,n}(f)$ and $I_{\cP}(f)$. 
Although the above bound depends on $K$, in practice,
$n=30$ would be sufficient for typical QMC integration,
since QMC integration error would be much larger than
$K2^{-30}\sqrt{S}$. 
Further justification for this discretization is that,
in single precision arithmetic 
real numbers in $I$
are usually discretized as a 23-digits binary fraction, 
and hence in practical 
integrations, $n\geq 24$ would be sufficiently large.

After the above discretization, our integration 
is exactly a finite sum (\ref{eq:integration-noerror}), 
and for a point set $P \subset V$, our discretized QMC
integration is (\ref{eq:discretized-QMC}). 
Hence we define the $n$-discretized QMC integration error by
the difference:
\begin{equation}\label{eq:discretized-QMC-error}
\Error(f;P,n):=|I_{P,n}(f) - I(f)|,
\end{equation}
which we are going to 
use as a proxy for (\ref{eq:QMC-error}) with the justification
above.

\begin{remark}\label{rem:midpoint}
In QMC-integration, it would be better to pick the mid point 
$\bx$ in the hypercube $I_\bb$ and use $f(\bx)$ 
as an approximation of the average $f_n(\bb)$ 
of $f$ over $I_\bb$, instead of using the corner point
$\bb$. Thus, we should translate $\cP$ by
adding $(2^{-n-1},2^{-n-1},\ldots,2^{-n-1})$. This is reflected
in our experiments in \S\ref{sec:experiments}. 
\end{remark}

\section{Discrete Fourier transformation}\label{sec:fourier}
\subsection{Preliminary}\label{sec:preliminary}
This section recalls well-known facts on discrete Fourier transformations.
Recall that $V$ is the set of $S \times n$ matrices with
components in $\F2$. A subset $P\subset V$ is an $\F2$-linear
subspace if $P$ contains the zero matrix and closed
under summation as $\F2$-matrices (i.e. componentwise sum
modulo 2).  We henceforth assume that $P$ is $\mathbb{F}_2$-linear.

We define a (standard) inner product by
$$
V \times V \to \F2, \quad B, A \mapsto (B | A):=
\sum_{1\leq T\leq S, 1\leq j\leq n} b_{T,j}a_{T,j},
$$
where the summation in the right hand side is modulo 2.
For an $\F2$-linear subspace $P\subset V$,
we define its {\em orthogonal} subspace
$$
P^\perp := \{A \in V \ | \ (B | A)=0 \mbox{ for all } B \in P\}.
$$
This is a linear subspace of $V$, and we have
\begin{equation} \label{eqn:dimension}
\dim P + \dim P^\perp = \dim V = nS. 
\end{equation}
We define
$$
\langle B | A \rangle := (-1)^{(B|A)} \in \{1,-1\}.
$$
The following lemma is standard (see \cite{SKRIGANOV01} 
for a more extensive theoretical treatment 
in the context of coding theory and uniform distribution).
\begin{lemma} Let $P\subset V$ be a linear subspace.
Then,
$$
\sum_{B \in P} \langle B|A \rangle =
\left\{
\begin{array}{cc}
|P| & \mbox{ if } A \in P^\perp \\
0   & \mbox{ if } A \notin P^\perp
\end{array}
\right.
$$
\end{lemma}
\begin{proof}
If $A \in P^\perp$, then $\langle B|A \rangle=(-1)^{(B|A)}=1$ for
all $B \in P$, which settles the first case.
If $A \notin P^\perp$, then the map $P \to \F2$
defined by $B \mapsto (B|A)$ is nonzero and $\F2$-linear.
This and the dimension formula of linear algebra implies that
those $B \in P$ with $(B|A)=0$ has dimension $\dim P - 1$,
thus the number of such $B$ is $|P|/2$, and hence
is the same with the number of $B \in P$ with $(B|A)=1$,
and the sum cancels out. 
This implies the second case.
\end{proof}

For a function
$$
f: V \to \bR, 
$$
its \emph{discrete Fourier transform} $\hat{f}:V \to \bR$ is
defined by 
$$
\hat{f}(A):= \frac{1}{|V|}
\sum_{B \in V} f(B)\langle B|A \rangle.
$$
We have the Fourier expansion formula
\begin{equation}\label{eq:disc-walsh}
f(B) = \sum_{A \in V} \hat{f}(A)\langle B|A \rangle.
\end{equation}
In fact, the right hand side is 
$$
\sum_{A \in V}\left(\frac{1}{|V|} \sum_{B' \in V} f(B') \langle B'|A \rangle \langle B|A \rangle\right).
$$
If we consider the summation over $A$ first,
the sum of $\langle B'|A\rangle \langle B|A \rangle=\langle B'-B|A \rangle$ is
0 if $B'-B \notin V^\perp = \{0\}$, and $|V|$ if
$B'-B=0$. Thus, this sum is $f(B)$.
By averaging (\ref{eq:disc-walsh}) over $P$, we have
the following well-known theorem 
(see \cite{SKRIGANOV01}).
\begin{theorem}\label{th:poisson}(Poisson summation formula)

Let $P\subset V=(\F2^n)^S$ be a linear subspace.
For any $f:V \to \bR$, we have the identity
$$
\frac{1}{|P|}\sum_{B \in P} 
f(B) = 
\sum_{A \in V} \hat{f}(A)
 \left[\frac{1}{|P|}\sum_{B \in P}\langle B | A\rangle \right]
= \sum_{A \in P^\perp} \hat{f}(A).
$$
\end{theorem}
We want to know the integration
$$
I(f)=\frac{1}{|V|}\sum_{B \in V} f_n(B) = \hat{f_n}(0).
$$
Recall that we choose our point set $P$
to be an $\F2$-linear subspace of $V$. 
Then, the discretized QMC
integration (\ref{eq:discretized-QMC}) of $f_n$ by $P$ 
is given by Theorem~\ref{th:poisson}:
\begin{equation}\label{eq:disc-QMC2}
\frac{1}{|P|}\sum_{B \in P} 
f_n(B) 
= \sum_{A \in P^\perp} \hat{f_n}(A).
\end{equation}
This implies that the integration error
of discretized QMC by $P$ is exactly
\begin{equation}\label{eq:exact-error}
\frac{1}{|P|}\sum_{B \in P} f_n(B) - \hat{f_n}(0)
=
\sum_{A \in P^\perp-\{0\}} \hat{f_n}(A),
\end{equation}
and an error bound is given by 
\begin{equation}\label{eq:integration-error}
\Error(f;P,n)=\left|\sum_{A \in P^\perp-\{0\}} \hat{f_n}(A)\right|
\leq \sum_{A \in P^\perp-\{0\}} |\hat{f_n}(A)|.
\end{equation}

\subsection{Estimation of the Fourier coefficients}\label{sec:WAFOM}
Once we have some estimation function $c(A)>0$ for
$A \in V$ such that
$$
|\hat{f_n}(A)| \leq C_f\cdot c(A)
$$
where $C_f$ is a constant depending only on $f$,
then by (\ref{eq:integration-error}) 
we have an upper estimate of the discretized QMC integration error:
\begin{equation}\label{eq:error-bound}
\Error(f;P,n) \leq C_f\cdot\sum_{A \in P^\perp-\{0\}} c(A).
\end{equation}
%
Dick gives such a bound
for $\alpha$-smooth functions 
(see \cite{Dick-MCQMC} for the definitions of $\alpha$-smooth functions
and the norm $||f||_\alpha$).
In our context, the integer $\alpha$ is set to $n$.
The following is a very special case of $\mu_\alpha(\bk)$
defined in \cite[\S4.1]{Dick-MCQMC} (the case when 
the base is 2, each coordinate of the integer vector
$\bk$ has a binary expansion of at most $n$ digits, and
$n=\alpha$).
Note that
\cite[Propagation rule (2)]{BALDEARUX-DICK-PILL}
(\cite[Propagation rule (ii)]{DICK-KRITZER})
states that a higher order net which achieves an improved
rate of convergence of the integration error
for a given $\alpha$ achieves an improved rate for all 
$1\leq \alpha' \leq \alpha$. Hence it may be 
possible that an improved rate of convergence
can also be achieved for $\alpha < n$.

\begin{definition}\label{def:mu}
For $A=(a_{T,j})_{1\leq T \leq S, 1\leq j \leq n} \in V$,
define 
$$
\mu(A):= \sum_{1\leq T\leq S,1\leq j\leq n} j\times a_{T,j}.
$$
Note that here each $a_{T,j}\in \{0,1\}$ is considered as an integer,
not an element of $\F2$.
\end{definition}

The following theorem is a part of the results of \cite[\S4.1]{Dick-MCQMC}.
\begin{theorem}(Josef Dick)\label{th:estimation-by-mu}
Let $f$ be an $n$-smooth function.
There is a constant $C_{S,n}$ depending only on $n$ and $S$
such that
$$
|\hat{f_n}(A)| \leq C_{S,n} ||f||_n \cdot 2^{-\mu(A)}.
$$
\end{theorem}
\begin{proof}
This is contained in Dick's theorem. We use the following two facts:
(1) the Walsh coefficient $\hat{f}(\bk)$ in \cite{Dick-MCQMC}
coincides with $\hat{f_n}(\bk)$ when the number of binary digits of each
component
of $\bk$ do not exceed $n$, because $f_n$ is the truncation of the Walsh
series upto the degree $n$, (2) $\hat{f_n}(\bk)=\hat{f_n}(A)$
holds for $A$ being the matrix obtained as coefficients of binary expansion of
$\bk$, by definition of Walsh functions.
\end{proof}
As in \cite[\S4.2]{Dick-MCQMC}, this theorem and
the formula (\ref{eq:integration-error}) give the error bound
$$
\Error(f;P,n) \leq C_{S,n} ||f||_n \cdot \sum_{A \in P^\perp -\{0\}}2^{-\mu(A)}.
$$
Thus, it is natural to define 
a kind of ``figure of merit'' (Walsh figure-of-merit
or WAFOM)
of the point set $P$: 
\begin{definition}\label{def:WAFOM} (WAFOM) 
$$
\WF(P):= 
\sum_{A \in P^\perp -\{0\}}2^{-\mu(A)}.
$$
\end{definition}
Hence we have 
\begin{equation}\label{eq:bound-by-WAFOM}
\Error(f;P,n) \leq C_{n,s} ||f||_n \cdot \WF(P),
\end{equation}
suggesting that finding $P$ with small 
values of $\WF(P)$ would, \emph{a priori}, have better integration performance. 

\begin{remark}
For several (continuous but not smooth) functions $f$, 
we numerically compute an approximate value of $\hat{f_n}(A)$ for small $n$ and $S$,
such as $n=5$ and $S=3$, for each $A$. For $2^{nS}$ distinct $A$ we
observe a tendency 
of $|\hat{f_n}(A)|$ being proportional to
$2^{-\mu(A)}$. 
(It is observed for many functions, 
including the non-differentiable function (\ref{eqn:asian_integral}) defined later.
On the other hand, for some special type of functions
such as the form of $f(\bx)=g_1(x_1)g_2(x_2)\cdots g_S(x_S)$,
the behaviour of $|\hat{f_n}(A)|$ is different.)
This supports 
the bound in Theorem~\ref{th:estimation-by-mu} and is evidence that the bound
(\ref{eq:bound-by-WAFOM}) may be adequate in practical applications for
some non-smooth functions 
(c.f. \cite{Dick-Walsh}), and may suggest a closer relation 
between the QMC integration error and WAFOM.
\end{remark}

So far, we have merely treated a discrete version of Dick's method.
Our new proposal here is to compute $\WF(P)$
by the formula (\ref{eq:WAFOM}) in the next section,
and find $P$ with small $\WF(P)$, by some random search of $P$.

\section{Inversion formula for $\WF(P)$}\label{sec:explicit}
For any $c:V \to \bR$, Theorem~\ref{th:poisson}
(with $P$ and $P^\perp$ interchanged) gives 
\begin{equation}\label{eq:poisson}
\frac{1}{|P^\perp|}\sum_{A \in P^\perp} c(A)
= 
\sum_{B \in P} \hat{c}(B).
\end{equation}

An explicit computation below shows the following.

\begin{theorem}\label{th:fourier-of-c}
Let $V$ be the set of $S\times n$ matrices with
coefficients in $\F2$, $B=(b_{T,j}) \in V$,
$A=(a_{T,j}) \in V$, 
and let $c(A):= 2^{-\mu(A)}$ as in Definition~\ref{def:mu}. 
Then
$$
\hat{c}(B)=
\frac{1}{|V|}
\prod_{1\leq T \leq S, 1\leq j \leq n} (1+(-1)^{b_{T,j}}2^{-j}).
$$
\end{theorem}

\begin{corollary}\label{cor:bound-by-WAFOM}
Let $P\subset V$ be a linear subspace.
We have
$$
\sum_{A \in P^\perp} c(A)
=|P^\perp|\sum_{B \in P}\hat{c}(B)
=
\frac{1}{|P|}
\sum_{B \in P}
\prod_{1\leq T \leq S, 1 \leq j\leq n} (1+(-1)^{b_{T,j}}2^{-j}).
$$

By subtracting $c(0)=1$, we have
\begin{equation}\label{eq:WAFOM}
\WF(P)=
\frac{1}{|P|}
\sum_{B \in P}
\left\{ 
\prod_{1\leq T \leq S, 1 \leq j\leq n} [(1+(-1)^{b_{T,j}}2^{-j})] -1 
\right\}.
\end{equation}
This is computable in $O(nSN)$ steps of arithmetic
operations in real numbers, where $N=|P|$.
\end{corollary}

%
%
\begin{proof}
By definition of $\hat{c}$,
\begin{eqnarray*}
\hat{c}(B) & = & 
\frac{1}{|V|}\sum_{A \in V} c(A)\langle B|A \rangle =
\frac{1}{|V|}\sum_{A \in V} 2^{-\sum_{1\leq T\leq S, 1\leq j \leq n}ja_{T,j}}\langle B|A \rangle \\
& = & 
\frac{1}{|V|}\sum_{A \in V} 2^{-\sum_{1\leq T\leq S, 1\leq j \leq n}ja_{T,j}}
(-1)^{\sum_{1\leq T\leq S, 1\leq j \leq n}a_{T,j}b_{T,j}} \\
&=&
\frac{1}{|V|}\sum_{A \in V} 
\prod_{1\leq T\leq S, 1\leq j \leq n}[(-1)^{b_T,j}2^{-j}]^{a_{T,j}}\\
&=&
\frac{1}{|V|}
\prod_{1\leq T\leq S, 1\leq j \leq n}
  (1+(-1)^{b_T,j}2^{-j}) \\
\end{eqnarray*}
(the last equality translates a sum over $2^{nS}$ products 
 into a product of $nS$ of 2-term sums),
which proves Theorem~\ref{th:fourier-of-c}.
The first identity in Corollary follows 
from $|V|=|P|\cdot |P^\perp|$
and (\ref{eq:poisson}). Formula~(\ref{eq:WAFOM}) 
follows from this and Definition~\ref{def:WAFOM}.
\end{proof}

A merit of the formula (\ref{eq:WAFOM}) is that
the number of summation depends only on 
$|P|$, not $|P^\perp|=2^{nS-\dim P}$ 
as in Definition~\ref{def:WAFOM}.

For QMC, the size $|P|=2^{\dim P}$ can not
exceed a reasonable number of operations in a computer,
say, $\dim P\leq 40$, since we need to take average of $f$
over $P$. Thus, (\ref{eq:WAFOM}) is practically computable, 
unlike a na\"{i}ve computation of $\WF(P)$ from
Definition~\ref{def:WAFOM}, 
which requires an intractable $2^{nS -\dim P}$ additions
for moderate $n$ and $S$, say $Sn\geq 80$.

\begin{remark}
The star-discrepancy of a point set is
a standard measure of the uniformity of point sets,
but it is hard to compute. Thus, only constructions of
point sets which gives an upper bound on the star-discrepancy
are studied, without explicit value, such as $(t,m,s)$-nets.
In contrast,
the formula (\ref{eq:WAFOM}) makes
WAFOM computable, which enables a random search.

As a preceding study, we note that \cite{HONG} studied 
optimization of the $t$-value of $(t,m,s)$-nets by a randomized 
(evolutionary) algorithm, and often obtained better $t$-values
than the original point sets, such as Sobol point sets
and Niederreiter-Xing point sets.

Another type of measures of regularity of a point set is
diaphony \cite{Zinterhof}\cite[Exercise~5.27]{Kuipers} 
and dyadic diaphony \cite{Hellekalek},
which are computable in $O(SN^2)$ steps.

WAFOM is superior to diaphony
in the order of computational complexity, which is of practical 
interest since $O(N^2)$ steps would be difficult
for iterated random search for large $N$.
Moreover, $\WF(P)$ has a direct error bounding formula 
(\ref{eq:bound-by-WAFOM}).

We note that WAFOM can be defined only for digital nets,
while star-discrepancy and diaphony are measures 
of uniformity defined on any point sets.

As stated in the introduction,
in coding theory, it is often difficult
to construct a family of good codes explicitly, 
and a random search sometimes yield a better result.
WAFOM and Corollary~\ref{cor:bound-by-WAFOM} offer a framework for 
obtaining a good QMC-point set by random search in a similar fashion.
\end{remark}

\section{Some Numerical Evidence}\label{sec:experiments}
\subsection{Point set generators}\label{sec:sequential_generator}
WAFOM is defined for any linear subspace $P\subset V=(\F2^n)^S$.
For fixed $n, S$ and $d:=\dim P$, it is natural to search for $P$ randomly
by uniform choice of its basis.
In this study, we restrict our search to the point sets
generated by an $M$-sequence of a fixed primitive
polynomial of degree $d$, with uniform random search of 
the output transform matrix, which we briefly recall
(see \cite{golomb} for basics on $M$-sequences). We call this type of point set generator a {\em sequential generator}, for which WAFOM can be 
computed in $O(nN)$ steps (see Proposition~\ref{prop:O_nN} below).

Let $(a_1,\ldots,a_d) \in \F2^d$. 
A linear recurring sequence associated to $(a_1,\ldots,a_d)$
is a sequence $x_0, x_1, \ldots \in \F2$ defined by a
recursion in $\F2$:
\begin{equation}\label{eq:linear-rec}
x_{j+d}=a_1x_{j+d-1}+\cdots + a_d x_j,
\end{equation}
with the initial vector $(x_0,\ldots,x_{d-1})$ given.
If the polynomial $t^d+a_1t^{d-1}+\cdots+a_d$ is primitive,
then the period of the sequence is $2^d-1$ for any 
nonzero initial vector. Such a sequence is called 
an $M$-sequence with degree $d$ and exists for any $d$.
The set of all the solutions of 
(\ref{eq:linear-rec}) constitutes an $\F2$-vector space
of dimension $d$, since the initial vector determines
the sequence.
Let us fix a nonzero initial vector. 
By the condition on the period, 
$(x_i,x_{i+1},\ldots,x_{i+d-1})$ is nonzero and distinct 
for $0\leq i \leq 2^d-1$, and hence assumes all nonzero
$d$-dimensional vectors. Thus, the zero sequence and 
$(x_i,x_{i+1},\ldots)$ for $0\leq i \leq 2^d-1$ give 
all of the solutions to (\ref{eq:linear-rec}), which constitute
a $d$-dimensional $\F2$-vector space.

For an integer $S$ and $0\leq k \leq 2^d-1$,
define an $S\times d$ matrix $C_k$ whose
$(T,j)th$ entry is $x_{k+j+T-2}$,  
that is: the first row of $C_k$ is $(x_k, x_{k+1},\ldots,x_{k+d-1})$,
the second row is $(x_{k+1}, x_{k+2},\ldots,x_{k+d})$, etc.,
and the last row is $(x_{k+S-1}, x_{k+S},\ldots,x_{k+S+d-2})$.
By the above observation, the set
$$
W:=\{C_k \ | \ 0 \leq k \leq 2^d-1\} \cup \{0\}
$$
is a $d$-dimensional sub vector space of
the space of $S\times d$ matrices.

To obtain a $d$-dimensional subspace $P \subset V$, 
we choose a $d\times n$ matrix $U$ of rank $d$,
and compute the image $WU$ of $W$ in $V$ by
the multiplication of $U$ from the right.
In other words, we compute $S \times n$ matrices
$C_kU$ for each $k$ ($0\leq k\leq 2^d-1$) 
and append the $0$ matrix to obtain $WU \subset V$.
We use $WU$ for our point set $P$. By changing $U$, 
we have various $P$.

An advantage of such a construction is in the efficiency 
for generating points. Because the first $S-1$ rows
of $C_{k+1}U$ are the same with the last $S-1$ rows
of $C_kU$, to compute the $(k+1)$-st point,
we need only to compute the last row of $C_{k+1}U$.
To do this, we compute
$(x_{k+S}, x_{k+S+1},\ldots,x_{k+S+d-1})$
by (\ref{eq:linear-rec}) and multiply by $U$ from the right.

Another advantage of sequential generators is 
in computing WAFOM.
In Theorem~\ref{th:fourier-of-c}, 
we need to compute 
$$
\hat{c}(B)
=\prod_{1\leq T \leq S}\left[\prod_{1\leq j \leq n} (1+(-1)^{b_{T,j}}2^{-j})\right].
$$
When $\bb_T$ denotes the $T$-th row of $B$, this can be written as
$$
\hat{c}(B)=\prod_{1\leq T \leq S}\left[\prod_{1\leq j \leq n} h(\bb_T)\right],
$$
where $h$ is defined in an obvious fashion. 
Since $C_{k+1}U$ is obtained from $C_kU$ by 
removing the first row and attaching the last row, 
we may obtain $\hat{c}(C_{k+1}U)$ by simply 
multiplying $\hat{c}(C_kU)$ by 
$h(\mbox{attached row})/h(\mbox{removed row})$.
Hence, if we record the values of $h$ for past $S$ rows then only one multiplication and one division is necessary. Thus we obtain the following.
\begin{proposition}\label{prop:O_nN}
For a point set generated by a sequential generator, 
WAFOM is computable in $O(nN)$ steps.
\end{proposition}
In practice, this multiplication and division accumulate
the truncation error.  One way to avoid this would be to divide
the sequence into moderate length subsequences,
and to apply this trick for each subsequence.

\begin{remark}
We are not sure whether such a choice of point sets
harms our ability to attain a small value of WAFOM, 
compared to a random search of $P$ by basis.  
\end{remark}

\subsection{Finding good point sets}
\label{sec:finding}

For an application, it is necessary to have a $P$ with a small
$\WF(P)$. As shown in the following experiments, 
even a na\"{i}ve random search
turns out to find a good $P$.

According to 
Lemma~\ref{lemma:QMC_discretization} and the discussion there, 
we fix $n=30$. 
We focus on the integrand
function described in \S~\ref{sec:performance} with dimension $S=4$.
Our search for a good point set proceeded in two stages for ease. 
The range of the $\F2$-dimension of $P$ is
$d \in \{10, 11, \hdots, 22\}$.  Fix a $d$ in this range.
At the first stage, we generate a $d\times d$ matrix $U'$
of rank $d$ uniformly at random 5000 times. 
For each $U'$, we generate $WU'\subset (\F2^d)^S$ 
as in \S\ref{sec:sequential_generator},
then compute WAFOM of $WU'$ with degree of discretization $d$.
Thus, we compute 5000 WAFOMS, and identify the best $U'$.
At the second stage, we generate an $S \times (n-d)$ matrix
uniform at random, concatenate it from the right to this best $U'$
to obtain an $S \times n$ matrix $U$. Then we compute 
the WAFOM of $WU \subset (\F2^n)^S$. We iterate this
2000 times, and take the best $P=WU$ with respect 
to $\WF(P)$.

The coefficients of primitive polynomials $(a_1, \hdots, a_d)$ were 
generated at random, bit by bit, with a bias towards ones: 
the heuristic being that if many coefficients are zero 
then the point set will satisfy a linear relation with
a small number of terms, which may be harmful to 
the effectiveness of QMC.  
For each $d$, we used the same primitive 
polynomial of degree $d$ for all matrices $U$.

Figure~\ref{fig:WAFOM} plots the smallest WAFOM 
values found by this procedure for each $d$. 
An ordinary least squares estimates of the slope 
suggests that the best point sets in a random search 
achieve an order of convergence something like $O(N^{-1-\beta})$ 
where $\beta \approx 1$ in the range of $10\leq d \leq 22$. 
The first $2^d$ points of a Sobol sequence are a digital net 
over $\mathbb{F}_2$, and therefore form 
an $\mathbb{F}_2$-linear point set, 
such that we can compute the WAFOM; 
it is included as the dashed line for comparison.  
The Sobol sequence, as well as the 
Faure and Halton sequences plotted in Figure~\ref{fig:error}, 
come from the open source C++ library quantlib \cite{web:quantlib}.

\begin{figure}
  \begin{center}
      \includegraphics[width=0.8\textwidth]{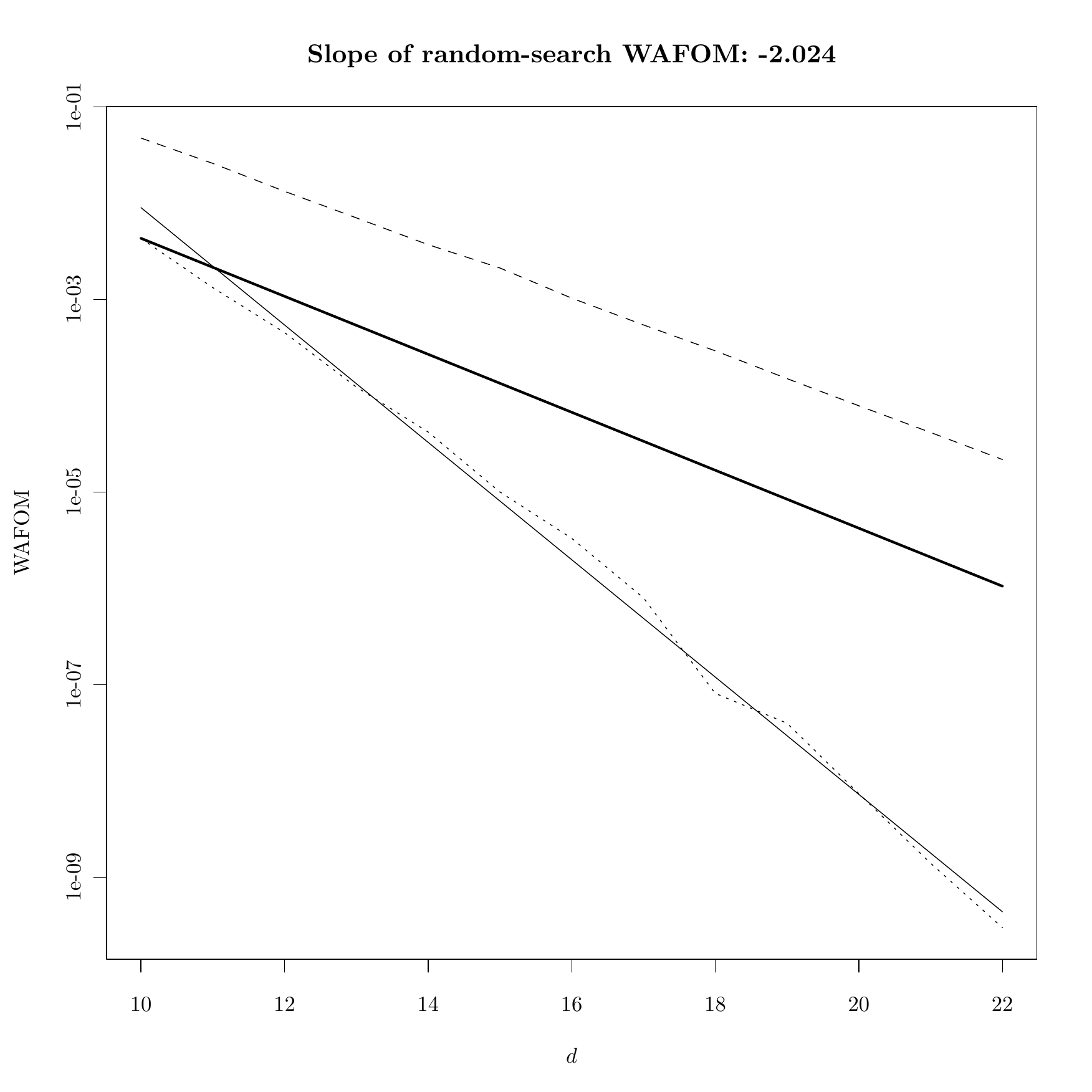}
    \caption{WAFOM (in log scale) for selected point sets.  The WAFOM of the Sobol sequence is the dashed line above and the dotted line below is the WAFOM of random search point sets.  The thin solid line estimates the rate of convergence.  The heavy black line has a slope of -1 for reference.  }
    \label{fig:WAFOM}
  \end{center}
\end{figure}

\begin{remark}

Dick \cite{Dick-MCQMC} has shown a construction of
a family of point sets with $\WF(P)=O(N^{-\alpha}(\log N)^{\alpha S})$ 
for arbitrary $\alpha>1$. This implies that, if we 
exhaustively search all of the linear point sets $P$ for every $d$ 
and plot their smallest WAFOM, then we can draw a line with slope $-\alpha$
that is above all the plotted points. 
The line of smallest WAFOM values in Figure~\ref{fig:WAFOM} seems to be convex above
(with an exception at $d=18$), which is in accordance with Dick's result.

For practical QMC, $N$ is
bounded above, and Figure~\ref{fig:WAFOM} might
show the practical upper bound of the value $\alpha$.
\end{remark}

\subsection{Performance in QMC integration} \label{sec:performance}

To illustrate the effectiveness of this method for improving the
efficiency of QMC integration, we consider a basic problem from
computational finance: the pricing of an Asian option.  Readers not interested in some context can skip directly to (\ref{eqn:asian_integral}),
where a problem is formulated directly in the form of (\ref{eqn:integral}).

The Asian option is a simple example of a \emph{path dependent option}, meaning that the payoff depends upon the path taken to the terminal point.  In this case, the payoff is some average price of the underlying until maturity.  This dependence on the value of the underlying at several points in time make its valuation a difficult problem, both analytically and computationally. For low dimensional or highly specialized
cases numerical solution of an associated partial differential equation is quite efficient.  For extensions beyond this, however, the preferred methodology for pricing Asian options is (quasi-) Monte Carlo simulation.



Here we endeavor to price an option on an asset having a geometric Brownian motion dynamics with a current price of $P_0$, (exponential) volatility $\sigma^2$, strike $K$, and maturity $T$ year.  The riskless discount rate is $r$.  Thus, the price of an Asian option with maturity $T$ sampled at times $Ti/S$, $i=1, \hdots, S$ is

\begin{align}
\label{eqn:asian}
e^{-rT} E^*\left[\left( \frac{1}{S} \sum_{i=1}^S P_{Ti/S}  - K \right)_+  \right],
\end{align}
where $E^*$ denotes expectation with the (logarithmic) drift of $P$ replaced by $r$ and $x_+ = \max(x,0)$.  Under the risk neutral measure, we can write
\begin{align}
P_t = P_0e^{\left(r - \frac{\sigma^2}{2} \right)t + \sigma W_t},
\end{align}
where $W_t$ is a standard Wiener process.  Inverting the Gaussian distribution, $\Phi$, we can map a quasirandom uniform $[0,1)$ variate to an increment of $P_t$ by

\begin{align}
x \mapsto P_0e^{\left(r - \frac{\sigma^2}{2} \right)t + \sigma \sqrt{t} \Phi^{-1}(x) }.
\end{align}
And to a $\mathbf{x} \in [0,1)^4$ we can associate a realization of
 (\ref{eqn:asian})
 using the independent increments property of Brownian motion.  In particular, for the purposes of QMC integration, we write (\ref{eqn:asian}) as:

\begin{align}
\label{eqn:asian_integral}
\int_{[0,1)^S} e^{-rT} \left( \frac{1}{S} \sum_{i=1}^S P_0e^{\left( r- \frac{\sigma^2}{2} \right)t + \sigma \sum_{j=1}^i \sqrt{Tj/S} \Phi^{-1}(x_j)} -K \right)_+ d \mathbf{x}.
\end{align}

As noted in Remark~\ref{rem:midpoint}, our point set is a subset of $\{0, 1/2^n, \hdots, 1-1/2^n\}^S$, so in our experiment we generate standard normal variates by $x \mapsto \Phi^{-1}(x + 1/2 \times 2^{-n})$ to center the point within each hypercube.  

The QMC integration errors are presented in Figure~\ref{fig:error}, where our random search method is seen to compare favorably with classical QMC point sets.  To evaluate the ``true'' option price, we performed QMC integration using a very large classical low discrepancy point set.

\begin{figure}
  \begin{center}
                       \includegraphics[width=0.8\textwidth]{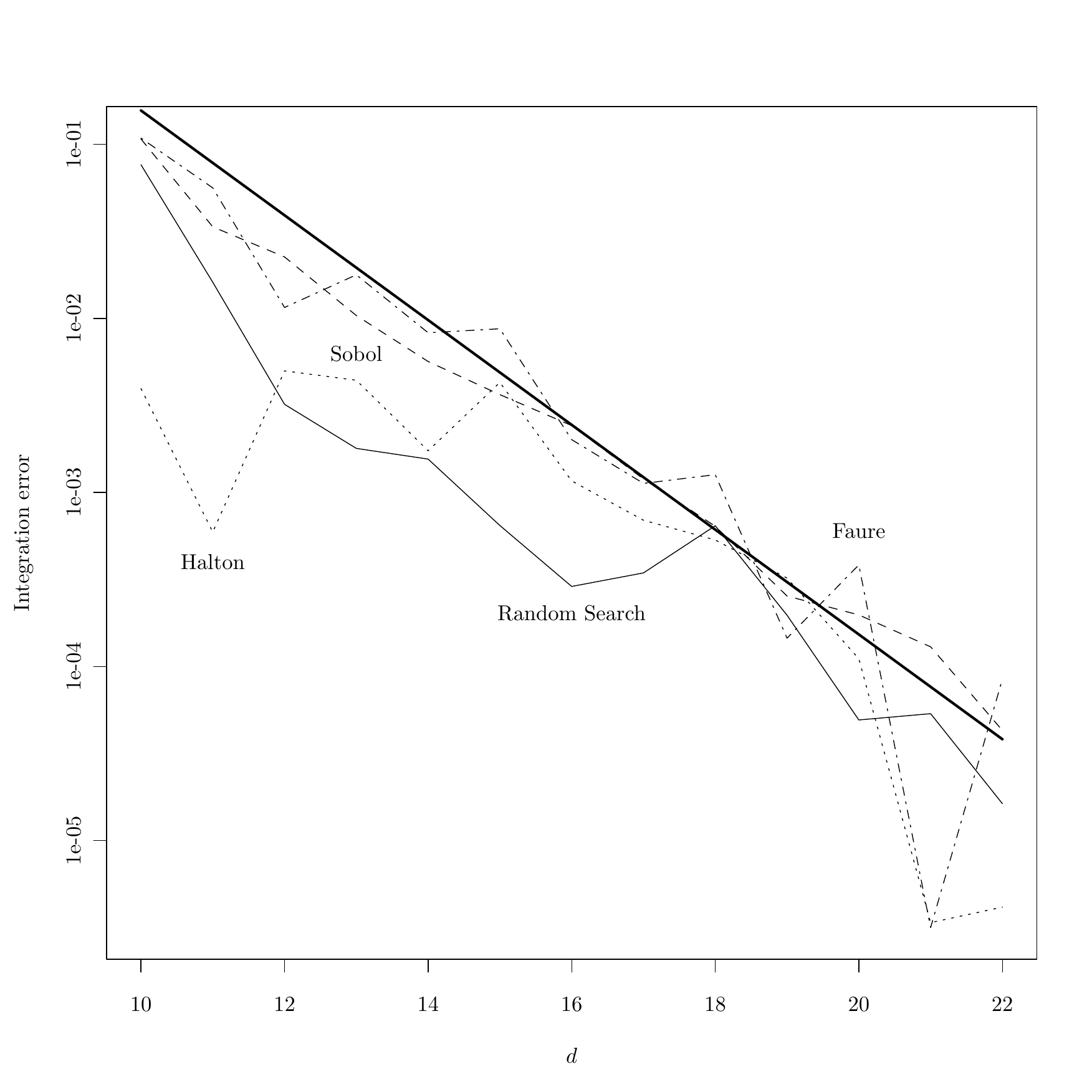}
    \caption{Integration error for the lowest WAFOM generators found by
   the procedure described in Section~\ref{sec:finding} (thin solid line).  
The error of several classical QMC point sets (Halton is dotted, Sobol is dashed, Faure is alternating dots and dashes) is included for comparison.  A heavy black line has slope of -1.}
    \label{fig:error}
  \end{center}
\end{figure}

\section{Conclusion}
We proposed Walsh figure-of-merit (WAFOM, Definition~\ref{def:WAFOM})
for those point sets $P$
constructed from $\F2$-linear spaces by digital nets, following some
earlier work by Dick \cite{Dick-Walsh, Dick-MCQMC}.

WAFOM is a measure of regularity of the point set, with a mathematical 
support by a Koksma-Hlawka type inequality (\ref{eq:bound-by-WAFOM})
due to Dick. The WAFOM of a
point set $P \subset (\F2^n)^S$ is quickly computable.
Because evaluating this measure requires only $O(nSN)$ operations
(\ref{eq:WAFOM}), 
a random search for generators having good WAFOM is a feasible
undertaking.
Applying even the crude method (\S\ref{sec:finding})
for $|P|=2^d$, 
$10\leq d\leq 22$, we found point sets with small WAFOM
(Figure~\ref{fig:WAFOM})
with order $O(N^{-1-\beta})$ for $\beta \approx 1$.

These point sets showed
excellent integration performance on 
an important problem for computational finance
(Figure~\ref{fig:error}),
even better 
than some standard QMC, namely, Sobol, Halton and Faure.
In the experiments, the integrand is not differentiable.
This is a supportive evidence of the effectiveness of WAFOM for
even non smooth functions, which is not covered by 
Theorem~\ref{th:estimation-by-mu}.

\subsection*{Acknowledgements}
The authors are thankful to Professor Niederreiter and 
two anonymous referees for valuable comments and 
for making us aware of some earlier literature, 
and to Professor Hickernell for letting us know \cite{HONG}.


\bibliographystyle{amsplain}

\bibliography{sfmt-kanren}


\end{document}